\documentclass{amsart}

\usepackage{amsmath,amssymb,amsthm,mathtools}
\usepackage{enumitem}
\usepackage{mathrsfs}
\usepackage{xcolor}
\usepackage{aliascnt}
\definecolor{refblue}{RGB}{0,27,126}
\definecolor{citegreen}{RGB}{0,126,68}
\usepackage[
  colorlinks=true,
  linkcolor=refblue,
  citecolor=citegreen,
  urlcolor=purple,
  filecolor=purple]{hyperref}
\hypersetup{pdftitle={Sharp shifted reciprocal sums of Neumann eigenvalues on space forms},pdfauthor={Daguang Chen, Chengxi Yang}}
\usepackage[nameinlink,capitalize,noabbrev]{cleveref}
\usepackage{microtype}
\usepackage[a4paper, left=2.8cm, right=2.8cm, top=2.5cm, bottom=2.5cm]{geometry}

\numberwithin{equation}{section}

\newtheorem{theorem}{Theorem}[section]

\newaliascnt{proposition}{theorem}
\newtheorem{proposition}[proposition]{Proposition}
\aliascntresetthe{proposition}

\newaliascnt{lemma}{theorem}
\newtheorem{lemma}[lemma]{Lemma}
\aliascntresetthe{lemma}

\newaliascnt{corollary}{theorem}
\newtheorem{corollary}[corollary]{Corollary}
\aliascntresetthe{corollary}

\newaliascnt{assumption}{theorem}

\aliascntresetthe{assumption}

\theoremstyle{definition}
\newaliascnt{definition}{theorem}

\aliascntresetthe{definition}

\newaliascnt{remark}{theorem}
\newtheorem{remark}[remark]{Remark}
\aliascntresetthe{remark}

\crefname{theorem}{Theorem}{theorems}
\Crefname{theorem}{Theorem}{Theorems}

\crefname{proposition}{Proposition}{propositions}
\Crefname{proposition}{Proposition}{Propositions}

\crefname{lemma}{Lemma}{lemmas}
\Crefname{lemma}{Lemma}{Lemmas}

\crefname{corollary}{Corollary}{corollaries}
\Crefname{corollary}{Corollary}{Corollaries}

\crefname{assumption}{Assumption}{assumptions}
\Crefname{assumption}{Assumption}{Assumptions}

\crefname{definition}{Definition}{definitions}
\Crefname{definition}{Definition}{Definitions}

\crefname{remark}{Remark}{remarks}
\Crefname{remark}{Remark}{Remarks}

\newcommand{\M}{\mathbb M}
\newcommand{\Sph}{\mathbb S}
\newcommand{\Hyp}{\mathbb H}
\newcommand{\R}{\mathbb R}
\newcommand{\1}{\mathbf 1}
\newcommand{\tr}{\operatorname{tr}}
\newcommand{\Id}{\mathrm I}

\newcommand{\Span}{\operatorname{span}}

\newcommand{\dd}{\,d}
\newcommand{\transpose}{\mathsf T}

\newcommand{\Om}{\Omega}
\newcommand{\sn}{\operatorname{sn}}

\title[Shifted reciprocal sums on space forms]
{Sharp shifted reciprocal sums of Neumann eigenvalues on space forms}

\author{Daguang Chen}
\email{dgchen@tsinghua.edu.cn}
\author{Chengxi Yang}
\email{ycx24@mails.tsinghua.edu.cn}
\address{Department of Mathematical Sciences, Tsinghua University, Beijing 100084, P. R. China.}


\thanks{The authors were supported by NSFC grant No. 11831005 and NSFC-FWO W2521103.}
\date{\today}

\begin{document}

\begin{abstract}
Let $\M_\kappa^n$ be the space form of sectional curvature $\kappa\in\{-1,0,1\}$, so that $\M_{-1}^n=\mathbb{H}^n, \M_{0}^n=\mathbb{R}^n, \M_{1}^n=\mathbb{S}^n$. 
Let $\Om\subset\M_\kappa^n$ be a nonempty bounded open set with Lipschitz boundary, and assume that $0<|\Omega|<|\mathbb{S}^n|$ when $\kappa=1$. 
Write $0=\mu_0(\Om)\leq\mu_1(\Om)\leq\cdots$ for the Neumann spectrum, and let $B_R^\kappa\subset\M_\kappa^n$ be a geodesic ball of volume $\vert\Om\vert/2$. 
We prove the sharp shifted reciprocal inequality
\[
    \sum_{j=2}^{n+1}\frac1{\mu_j(\Om)} \geq \frac{n}{\mu_1(B_R^\kappa)} = \frac{n}{\mu_2(B_R^\kappa\sqcup B_R^\kappa)}.
\]
Equality holds if and only if $\Omega$ is the disjoint union of two equal geodesic balls. This gives an affirmative answer to the conjecture of \cite[Remark~11]{BucurMartinetNahon2025}.
\end{abstract}

\maketitle

\section{Introduction}
Let $\Om$ be a bounded open set in an $n$-dimensional Riemannian manifold \((M,g)\). The Neumann problem is
\begin{equation*}
    -\Delta u=\mu u\quad\text{in }\Om, \qquad \partial_\nu u=0\quad\text{on }\partial\Om.
\end{equation*}
We enumerate the eigenvalues, with multiplicity, by
\[
    0=\mu_0(\Om)\leq\mu_1(\Om)\leq\mu_2(\Om)\leq\cdots.
\]
Thus $\mu_1$ is the first positive eigenvalue when $\Om$ is connected. If $\Om$ has two components, then $\mu_0=\mu_1=0$.  Throughout the paper the reciprocal of a zero eigenvalue is understood as $+\infty$.

\subsection{The Szeg\H{o}--Weinberger line of results}
The free-membrane isoperimetric problem begins with Szeg\H{o}'s theorem \cite{Szego1954} for simply connected planar domains. Besides the sharp bound for the first positive eigenvalue, his conformal transplantation yields the two-term reciprocal inequality
\[
    \frac1{\mu_1(\Om)}+\frac1{\mu_2(\Om)}\geq\frac2{\mu_1(D)},
\]
where $D$ is the disk of the same area. 
Weinberger \cite{Weinberger1956} removed both the simple connectivity assumption and the dimensional restriction, proving that the Euclidean ball maximizes $\mu_1$ under a volume constraint. The proof chooses a center for which the $n$ transplanted degree-one ball eigenfunctions have zero mean, and then compares their mass and energy by radial transplantation. 
This center-of-mass construction is the common ancestor of all the trial functions used below.

The Szeg\H{o}--Weinberger inequality was extended to curved spaces along two complementary lines. 
In dimension two, Bandle \cite{Bandle1972,Bandle1980} used conformal transplantation to treat simply connected surfaces with curvature bounded above, while Chavel \cite{Chavel1980,Chavel1984} adapted Weinberger's trial-function method to spaces of constant curvature. 
In hyperbolic space, Ashbaugh--Benguria \cite{AshbaughBenguria1995} and Xu \cite{Xu1995} established the sharp inequality without any geometric inclusion hypothesis. 
In positive curvature, Ashbaugh and Benguria \cite{AshbaughBenguria1995} proved that a geodesic ball uniquely maximizes the first positive Neumann eigenvalue among domains contained in a hemisphere. 
For simply connected domains in $\Sph^2$, this restriction was subsequently relaxed: Langford and Laugesen \cite{LangfordLaugesen2023} extended the result to domains whose area is at most $16/17$ of $\vert\Sph^2\vert$, and Provenzano and Savo \cite{ProvenzanoSavo2026} recently removed the area restriction altogether, proving that geodesic disks are the unique maximizers for every prescribed area.



\subsection{The Ashbaugh--Benguria reciprocal conjecture}
The first positive Neumann eigenvalue of an $n$-ball has multiplicity $n$. 
This suggests retaining the whole degree-one eigenspace rather than estimating one coordinate at a time. 
Ashbaugh and Benguria \cite{AshbaughBenguria1993,Ashbaugh1999} conjectured that, if $B_{\Om}\subset\R^n$ is a ball with $|B_{\Om}|=|\Om|$, then
\begin{equation*}
    \sum_{j=1}^{n}\frac1{\mu_j(\Om)} \geq \frac{n}{\mu_1(B_{\Om})},
\end{equation*}
with equality only for a ball. 
Their universal estimate had the correct scale but not the conjectured sharp constant. 
The first decisive progress was obtained by Xia and Wang \cite{WangXia2023}, who proved the sharp $(n-1)$-term estimate in Euclidean and hyperbolic spaces,
\begin{equation*}
    \sum_{j=1}^{n-1}\frac1{\mu_j(\Om)} \geq \frac{n-1}{\mu_1(B_{\Om})},
\end{equation*}
with rigidity for the ball. 
Benguria, Brandolini, and Chiacchio \cite{BenguriaBrandoliniChiacchio2020} proved on spheres an extension of the result of Xia and Wang for harmonic means, still under the constraint on $\Omega$ to be included in a hemisphere. Subsequent work extended this lower-order harmonic-mean estimate in several directions. 
Chen and Mao \cite{ChenMao2024} treated the Witten Laplacian on Euclidean and hyperbolic domains; Chen \cite{Chen2024} obtained a comparison under an upper sectional-curvature bound and a lower Ricci-curvature bound; Gao and Wang \cite{GaoWang2026} established a Gaussian-space analogue for origin-symmetric Lipschitz domains.

The conjecture itself was recently resolved in 2026. He, Li, and Tang \cite{HeLiTang2026} posted a proof of the full Euclidean inequality for smooth bounded domains. 
Shortly thereafter, You and Zhang \cite{YouZhang2026} proved the full $n$-term inequality in the non-Euclidean space forms: for bounded connected smooth domains in hyperbolic space, and for such domains contained in an open hemisphere of the sphere, the geodesic ball uniquely minimizes the reciprocal sum. 
Thus the one-ball Ashbaugh--Benguria problem is now complete in the smooth-domain class in the Euclidean and hyperbolic settings and in the spherical hemisphere regime. 
Beyond the sharp inequality and its equality characterization, Li and Wang \cite{LiWang2026} obtained a quantitative stability result for the Euclidean reciprocal-sum inequality, relating the spectral deficit to geometric asymmetry.

These one-ball results are distinct from the two-ball problem treated here, where the reciprocal sum starts at $\mu_2$ and the comparison ball has half the volume of the domain.

\subsection{From the second Neumann eigenvalue to the Bucur--Martinet--Nahon conjecture}

The optimizer changes at the second nontrivial eigenvalue. 
Girouard, Nadirashvili, and Polterovich \cite{GirouardNadirashviliPolterovich2009} proved in the simply connected planar class that the sharp bound is approached by domains degenerating to two equal disks.
Bucur and Henrot \cite{BucurHenrot2019} removed the dimensional and topological restrictions and proved, also in a relaxed density class, that two disjoint equal Euclidean balls maximize the second nontrivial Neumann eigenvalue.
Their trial field glues two Weinberger fields across a folding hyperplane and a degree argument imposes orthogonality both to constants and to a first Neumann mode.

Freitas and Laugesen \cite{FreitasLaugesen2022} extended the theorem to hyperbolic space and gave a new proof of the Euclidean orthogonality step. 
Their argument uses a unique center of the folded measure, a canonical family of Euclidean or M\"obius translations, and Petrides's reflection-degree theorem \cite{Petrides2014}. 

On the unit sphere, Bucur, Martinet, and Nahon \cite{BucurMartinetNahon2025} proved that every open Lipschitz set, without any hemisphere-containment assumption, satisfies
\begin{equation}\label{eq:BMN-nminus1-intro}
    \sum_{j=2}^{n}\frac1{\mu_j(\Om)} \geq \frac{n-1}{\mu_1(B_R)}, \qquad \vert B_R\vert=\frac{\vert\Om\vert}{2},
\end{equation}
with equality if and only if $\Om$ is the disjoint union of two equal geodesic balls. 
Their proof involves a genuinely spherical topological argument: since an appropriate center need not be unique, they count the zeros of an equivariant vector field modulo four. 
They also stated in \cite[Remark~10]{BucurMartinetNahon2025}, that the corresponding result holds in the Euclidean space; see also \cite{Eddaoudi2025}.

In \cite[Remark~11]{BucurMartinetNahon2025}, Bucur, Martinet and Nahon asked whether the missing term in
\eqref{eq:BMN-nminus1-intro} can be restored:
\begin{equation*}
    \sum_{j=2}^{n+1}\frac1{\mu_j(\Om)} \geq \frac{n}{\mu_1(B_R)}, \qquad \vert B_R\vert=\frac{\vert\Om\vert}{2},
\end{equation*}
first on $\Sph^n$ and, analogously, in $\R^n$. 
This conjecture is supported by numerical computations on $\Sph^2$  \cite{Martinet2024}.
The constant is dictated by the model configuration consisting of two disjoint equal balls: its zero eigenspace is two-dimensional, while the next eigenvalue is $\mu_1(B_R)$ and has multiplicity $2n$. 
The same model configuration, together with the theorem of Freitas--Laugesen \cite{FreitasLaugesen2022}, naturally suggests the corresponding hyperbolic inequality. 

\subsection{Main result}
For $\kappa\in\{-1,0,1\}$ let $\M_\kappa^n$ be the space form of curvature $\kappa$:
\[
        \M_{-1}^n=\Hyp^n,\qquad \M_0^n=\R^n,\qquad \M_1^n=\Sph^n.
\]
The hyperbolic metric is normalized to curvature $-1$, and $\Sph^n$ is the unit sphere. 

The main result of this paper proves the conjecture of Bucur, Martinet and Nahon \cite[Remark~11]{BucurMartinetNahon2025}.

\begin{theorem}\label{thm:main}
Let $n\geq2$ and $\kappa\in\{-1,0,1\}$. Let $\Om\subset\M_\kappa^n$ be a nonempty bounded open set with Lipschitz boundary. In the spherical case assume $0<\vert\Om\vert<\vert\Sph^n\vert$. Let $B_R^\kappa\subset\M_\kappa^n$ be a geodesic ball of radius $R$ satisfying
\begin{equation*}
    \vert B_R^\kappa\vert = \frac{\vert\Om\vert}{2}.
\end{equation*}
Then
\begin{equation}\label{eq:main}
    \sum_{j=2}^{n+1}\frac1{\mu_j(\Om)} \geq \frac{n}{\mu_1(B_R^\kappa)} = \frac{n}{\mu_2(B_R^\kappa\sqcup B_R^\kappa)}.
\end{equation}
Equality holds if and only if $\Om$ is the disjoint union of two geodesic balls, each of volume $\vert\Om\vert/2$.
\end{theorem}
The proof of \cref{thm:main} can be formulated through global matrix comparisons. Rather than disintegrating the relevant integrals along individual geodesic rays and applying one-dimensional mass transplantation ray by ray, we retain the global integral representations of the mass and energy matrices and obtain the required estimates directly. This formulation also provides a convenient framework for the analysis of the equality case.

\begin{remark}\label{rem:harmonic-mean}
    Define the harmonic mean of the eigenvalue cluster $\mu_2(\Om),\ldots,\mu_{n+1}(\Om)$ by
    \begin{equation*}
        H_{2,n+1}(\Om)=\frac{n}{\sum_{j=2}^{n+1}\mu_j(\Om)^{-1}}.
    \end{equation*}
    Then \eqref{eq:main} is equivalently written as
    \begin{equation*}
        H_{2,n+1}(\Om)\leq\mu_1(B_R^\kappa).
    \end{equation*}
    Equality holds if and only if $\Om$ is the disjoint union of two geodesic balls, each of volume $\vert\Om\vert/2$.
\end{remark}

Since $\mu_j(\Om)\geq\mu_2(\Om)$ for $j=2,\ldots,n+1$, the reciprocal sum in \eqref{eq:main} is bounded above by $n/\mu_2(\Om)$. Hence \cref{thm:main} immediately yields the following sharp bound for the second nontrivial Neumann eigenvalue. The Euclidean case of this estimate was first established by Bucur and Henrot \cite{BucurHenrot2019}, while its hyperbolic and spherical counterparts were subsequently proved by Freitas and Laugesen \cite{FreitasLaugesen2022} and Bucur, Martinet and Nahon \cite{BucurMartinetNahon2025}, respectively.
\begin{corollary}\label{cor:second-Neumann-eigenvalue}
    Under the assumptions of \cref{thm:main},
    \begin{equation*}
        \mu_2(\Om)\leq\mu_1(B_R^\kappa).
    \end{equation*}
    Equality holds if and only if $\Om$ is the disjoint union of two geodesic balls, each of volume $\vert\Om\vert/2$.
\end{corollary}

When $\Om=\Om_1\sqcup\Om_2$ has exactly two connected components, the zero eigenvalue has multiplicity two, and hence
\begin{equation*}
    \mu_2(\Om)=\min\bigl\{\mu_1(\Om_1),\mu_1(\Om_2)\bigr\}.
\end{equation*}
Thus \cref{cor:second-Neumann-eigenvalue} admits the following natural componentwise formulation.
\begin{corollary}\label{cor:two-component-domains}
    Under the assumptions of \cref{thm:main}, suppose that $\Om=\Om_1\sqcup\Om_2$, where $\Om_1$ and $\Om_2$ are the two connected components of $\Om$. Then
    \begin{equation*}
        \min\bigl\{\mu_1(\Om_1),\mu_1(\Om_2)\bigr\}\leq\mu_1(B_R^\kappa), \qquad \vert B_R^\kappa\vert=\frac{\vert\Om\vert}{2}.
    \end{equation*}
    Equality holds if and only if $\Om_1$ and $\Om_2$ are equal geodesic balls.
\end{corollary}
A particularly useful specialization arises when the two components are isometric. Their first positive Neumann eigenvalues and volumes then coincide, so \cref{cor:two-component-domains} reduces to a one-domain estimate.

\begin{corollary}\label{cor:disjoint-isometric-copies}
    Under the assumptions of \cref{thm:main}, suppose that $\Om=\Om_1\sqcup T(\Om_1)$, where $T$ is an isometry of $\M_\kappa^n$ and the two copies are disjoint. Then
    \begin{equation*}
        \mu_1(\Om_1)\leq\mu_1(B_R^\kappa), \qquad \vert B_R^\kappa\vert=\vert\Om_1\vert.
    \end{equation*}
    Equality holds if and only if $\Om_1$ is a geodesic ball.
\end{corollary}
Thus the corresponding Szeg\H{o}--Weinberger inequality is recovered within this symmetric two-component setting.

Applying the power-mean inequality to the reciprocals of the same eigenvalue cluster gives a further consequence.
\begin{corollary}\label{cor:higher-reciprocal-moments}
    Under the assumptions of \cref{thm:main}, for every $s\geq1$,
    \begin{equation*}
        \sum_{j=2}^{n+1}\frac1{\mu_j(\Om)^s} \geq \frac{n}{\mu_1(B_R^\kappa)^s}.
    \end{equation*}
    Equality holds if and only if $\Om$ is the disjoint union of two geodesic balls, each of volume $\vert\Om\vert/2$.
\end{corollary}
Beyond these finite reciprocal moments, Freitas and Laugesen \cite[Section~1.3]{FreitasLaugesen2022} observed that, in the Euclidean and hyperbolic settings, it remains open whether, for $s>n/2$, the Neumann spectral zeta function
\[
    \sum_{j=1}^{\infty}\frac{1}{\mu_j(\Om)^{s}}
\]
is minimized by a ball among domains of prescribed volume.

The paper is organized as follows. 
In \cref{sec:Space-form notation and radial eigenfunctions}, we introduce unified polar notation for the three space forms and record the properties of the extended degree-one radial trial field. 
In \cref{sec:Variational and matrix tools}, we collect the shifted trace-form Ritz principle and the trace-free matrix inequality.
In \cref{sec:Folding constructions on space forms}, we construct folded and centered trial fields in the Euclidean, hyperbolic and spherical spaces, respectively. 
The proof of \cref{thm:main} is given in \cref{sec:Proof of main theorem}, and the equality case is completed in \cref{sec:Equality and geometric rigidity}.

\section{Space-form notation and radial eigenfunctions} \label{sec:Space-form notation and radial eigenfunctions}
For $\kappa\in\{-1,0,1\}$, let $\Om\subset\M_\kappa^n$ be a nonempty bounded open set with Lipschitz boundary, and assume that $0<|\Omega|<|\mathbb{S}^n|$ when $\kappa=1$. The embedding $H^1(\Om)\hookrightarrow L^2(\Om)$ is compact, and the Neumann quadratic form has a discrete spectrum
\[
    0=\mu_0(\Om)\leq\mu_1(\Om)\leq\mu_2(\Om)\leq\cdots\longrightarrow\infty.
\]
If $\mu_2(\Om)=0$, then the theorem is immediate. Hence, from now on, we may assume
\begin{equation*}
        \mu_2(\Om)>0.
\end{equation*}
Under this assumption $\Om$ has at most two connected components. 

Fix an origin $o\in\M_\kappa^n$.  Throughout the paper, geodesic polar coordinates about $o$ are denoted by
\begin{equation*}
    x=\exp_o(r\theta),\qquad (r,\theta)\in(0,\mathcal{L}_\kappa)\times\Sph^{n-1},\qquad \mathcal{L}_\kappa=
    \begin{cases}
        \ +\infty,&\kappa\leq0,\\
        \ \pi,&\kappa=1.
    \end{cases}
\end{equation*}
Thus $r=d(o,x)$ is the Riemannian distance from the origin and $\theta\in\Sph^{n-1}\subset T_o\M_\kappa^n$ is the angular variable. Define
\begin{equation*}
    \sn_\kappa(r)=
    \begin{cases}
        \ \sinh r,&\kappa=-1, \\
        \       r,&\kappa=0,  \\
        \  \sin r,&\kappa=1.
    \end{cases}
\end{equation*}
The Riemannian metric, volume element and Laplacian have the unified form
\begin{equation*}
    \begin{aligned}
        g_\kappa&=\dd r^2+\sn_\kappa(r)^2g_{\Sph^{n-1}}, \\
        \dd\nu&=\sn_\kappa(r)^{n-1}\dd r\dd\sigma(\theta),\\
        \Delta&=\partial_r^2+(n-1)\frac{\sn_\kappa'(r)}{\sn_\kappa(r)}\partial_r+\frac1{\sn_\kappa(r)^2}\Delta_{\Sph^{n-1}}.
    \end{aligned}
\end{equation*}
Let $B_r^\kappa\subset\M_\kappa^n$ be a geodesic ball of radius $r$. 
Let
\begin{equation*}
    L_\kappa=
    \begin{cases}
        \ +\infty, &\kappa\leq0,\\
        \ \pi/2,   &\kappa=1.
    \end{cases}
\end{equation*}
For convenience, we write
\begin{equation*}
    B^\kappa_{L_\kappa}=
    \begin{cases}
        \ \M_\kappa^n,&\kappa\leq0,\\
        \ B_{\pi/2}^1,&\kappa=1.
    \end{cases}
\end{equation*}
Thus \(B^\kappa_{L_\kappa}\) is precisely the polar region on which the global matrix comparisons will be performed.
Let $R$ be the comparison radius determined by $\vert B_R^\kappa \vert = \vert\Omega\vert/2$. Thus, $R<L_\kappa$. Put
\begin{equation*}
    \lambda=\mu_1(B_R^\kappa).
\end{equation*}
By separation of variables, the first nonzero Neumann eigenspace of $B_R^\kappa$ is the degree-one spherical-harmonic sector. More precisely, it has multiplicity $n$ and is spanned by
\begin{equation*}
    G_i(r,\theta)=g(r)\theta_i,\qquad i=1,\dots,n,
\end{equation*}
where $\theta_i$, $i=1,\dots,n$, are the coordinate functions of $\theta\in\Sph^{n-1}\subset\mathbb{R}^n$. The radial factor $g$ is uniquely determined up to multiplication by a positive constant and satisfies
\begin{equation}\label{eq:radial-ode}
    g''+(n-1)\frac{\sn_\kappa'}{\sn_\kappa}g'+\left(\lambda-\frac{n-1}{\sn_\kappa^2}\right)g=0, \qquad g(0)=0,\qquad g'(R)=0.
\end{equation}
Moreover,
\begin{equation*}
    g(r)>0 \quad \text{and} \quad g'(r)>0, \qquad 0<r<R.
\end{equation*}
For the Euclidean and hyperbolic cases, these facts are recorded in \cite[Propositions~A.1 and~A.2]{FreitasLaugesen2022}; see also \cite[p.~562]{AshbaughBenguria1993} for the Euclidean case and \cite[Sections~3 and~6]{AshbaughBenguria1995} for the hyperbolic case. For spherical caps, see \cite[Section~3]{BucurMartinetNahon2025}.

We use the same symbol $g$ for its Weinberger extension to the full polar range:
\begin{equation}\label{eq:g-extension}
    g(r)=
    \begin{cases}
        \ g(R),&\kappa\leq0,\quad r\geq R,\\
        \ g(R),&\kappa=1,\quad R\leq r\leq\pi-R,\\
        \ g(\pi-r),&\kappa=1,\quad \pi-R\leq r\leq\pi.
    \end{cases}
\end{equation}
In the last line the right-hand side is the solution already defined on $[0,R]$.  In particular, the spherical extension satisfies
\begin{equation*}
    g(\pi-r)=g(r).
\end{equation*}
The two radial quantities entering the mass and energy matrices have opposite monotonicity properties, as recorded below.

\begin{lemma}\label{lem:radial-monotonicity}
For each $\kappa$, 
\[
    g(r)^2\quad\text{is non-decreasing on }[0,L_\kappa),
\]
whereas
\[
    \frac{g(r)^2}{\sn_\kappa(r)^2}\quad\text{is strictly decreasing on }(0,L_\kappa).
\]
If $\kappa=1$, then on the interval $[\pi/2,\pi]$, $g^2$ is non-increasing and $g^2/\sn_\kappa^2$ is strictly increasing by the symmetries $g(\pi-r)=g(r)$ and $\sin(\pi-r)=\sin r$.
\end{lemma}
\begin{proof}
For the Euclidean and hyperbolic cases, these properties follow from \cite[Sections~3 and~6]{AshbaughBenguria1995}, \cite[Lemma~2]{Xu1995}, and \cite[Propositions~A.1 and~A.2]{FreitasLaugesen2022}. 
For the spherical case, see \cite[Section~3, pp.~53--54]{BucurMartinetNahon2025}. 
The assertions on $[\pi/2,\pi]$ follow immediately from the stated symmetries.
\end{proof}

Define the common degree-one radial field centered at $o$ by
\begin{equation}\label{eq:radial-field}
    \mathcal{G}(r,\theta)=g(r)\theta \in T_o\M_\kappa^n\simeq\R^n,
\end{equation}
whose $n$ components are $\{G_i\}_{i=1}^n$. Direct computation yields
\begin{equation}\label{eq:GG-and-nablaGnablaG}
    \begin{aligned}
        &\mathcal{G}\mathcal{G}^\transpose = \big(G_iG_j\big)_{ij} = g(r)^2\theta\theta^\transpose, \\
        &\langle\nabla\mathcal{G},\nabla\mathcal{G}^\transpose\rangle = \Big(\langle\nabla G_i,\nabla G_j\rangle\Big)_{ij} = g'(r)^2\theta\theta^\transpose+\frac{g(r)^2}{\sn_\kappa(r)^2}\Big(\Id-\theta\theta^\transpose\Big).
    \end{aligned}
\end{equation}

Set
\begin{equation*}
    \begin{aligned}
        A_R:=\int_{B^\kappa_R} g(r)^2 \dd\nu, \qquad Q_R:=\int_{B^\kappa_R} g'(r)^2 \dd\nu, \qquad H_R:=\int_{B^\kappa_R} \frac{g(r)^2}{\sn_\kappa(r)^2} \dd\nu.
    \end{aligned}
\end{equation*}
Since the restriction of each \(G_i\) to \(B_R^\kappa\) is an eigenfunction with eigenvalue \(\lambda=\mu_1(B_R^\kappa)\), we have the identity:
\begin{equation}\label{eq:energy-identity}
    Q_R+(n-1)H_R=\lambda A_R.
\end{equation}

\section{Variational and matrix tools}\label{sec:Variational and matrix tools}
The proof uses two finite-dimensional statements. The first is the shifted version of Hersch's trace variational principle \cite{Hersch1961,HileXu1993}.  The second is the trace-free convexity lemma that closes the He--Li--Tang argument \cite{HeLiTang2026}.

Let $u_0(={\rm{const.}})$ and $u_1$ be orthonormal eigenfunctions corresponding to $\mu_0(\Omega)$ and $\mu_1(\Omega)$ respectively.
For real symmetric matrices, we write $\mathbf A\succeq\mathbf B$ when $\mathbf A-\mathbf B$ is positive semidefinite and $\mathbf A\succ\mathbf B$ when $\mathbf A-\mathbf B$ is positive definite.

\begin{lemma}\label{lem:shifted-ritz}
    Let $P_1,\dots,P_n\in H^1(\Omega)$ be linearly independent in $L^2(\Omega)$ and satisfy
    \[
        \int_\Omega P_i u_0 \dd\nu = \int_\Omega P_i u_1 \dd\nu = 0, \qquad i=1,\dots,n.
    \]
    Set
    \[
        \mathbf M_{ij}=\int_\Omega P_iP_j \dd\nu,\qquad\mathbf K_{ij}=\int_\Omega\langle\nabla P_i,\nabla P_j\rangle \dd\nu.
    \]
    If $\mathbf K$ is positive definite, then
    \begin{equation*}
        \sum_{j=2}^{n+1}\frac1{\mu_j(\Omega)}
        \geq \tr(\mathbf K^{-1}\mathbf M).
    \end{equation*}
\end{lemma}
\begin{proof}
    This is the shifted version of \cite[Lemma~2.1]{HeLiTang2026}, obtained by applying the same Ritz--Poincaré argument to the spectral subspace $\{u_0,u_1\}^{\perp}$.
\end{proof}

The following lemma is \cite[Lemma~2.2]{HeLiTang2026}.
\begin{lemma}\label{lem:trace-free-matrix-Jensen-inequality}
    Let $a,c,d,\lambda>0$, let $\mathbf Z$ be a real symmetric $n\times n$ matrix with $\tr\mathbf Z=0$, and suppose
    \begin{equation*}
        \mathbf M\succeq \mathbf 0,\qquad\mathbf M\succeq a\Id+c\mathbf Z,\qquad\mathbf 0\prec\mathbf K\preceq\lambda a\Id-d\mathbf Z.
    \end{equation*}
    Then
    \begin{equation*}
        \tr(\mathbf K^{-1}\mathbf M)\geq\frac n\lambda.
    \end{equation*}
    If equality holds, then $\mathbf Z=\mathbf 0$, $\mathbf M=a\Id$, and $\mathbf K=\lambda a\Id$.
\end{lemma}

\section{Folding constructions on space forms}\label{sec:Folding constructions on space forms}
In this section, we record the topological inputs that produce an $n$-dimensional trial space orthogonal to $u_0$ and $u_1$ on $\M_\kappa^n$. The nonpositive and positive curvature constructions are different at the topological level, but they lead to the same analytic object: a fold density bounded by $2$.

Let $H$ be a geodesic half-space in $\M_\kappa^n$ and $R_H$ be reflection across its totally geodesic boundary. Define the fold map onto $H$ by
\begin{equation}\label{eq:general-fold}
    F_H(x)=
    \begin{cases}
        \ x,&x\in H,\\
        \ R_H(x),&x\notin H.
    \end{cases}
\end{equation}
The map $F_H$ is continuous, locally Lipschitz, and an isometry on each side of $\partial H$. It maps $\M_\kappa^n$ onto $\overline H$. Since $\partial H$ is a null set, we shall omit $\partial H$ in the sequel whenever no ambiguity arises.

Define
\begin{equation}\label{eq:folded-measure}
    \alpha_H := \1_{T_{-H}(\Omega\cap H)} + \1_{T_{-H}(R_H(\Omega)\cap H)}
\end{equation}
to be the fold density, where $$R_H(\Omega)\cap H = R_H\big(\Omega\cap(\M_\kappa^n\setminus\overline{H})\big).$$ In particular, $\alpha_H$ has compact support with $0\leq\alpha_H\leq2$.
Here $T_{-H}$ is a centering isometry, that is, a translation on $\mathbb{R}^n$, a M\"{o}bius transformation on $\mathbb{H}^n$, and a rotation on $\mathbb{S}^n$. 
Since $R_H, T_{-H}$ are isometries and $F_H$ is an isometry on each side of $\partial H$, we have 
\begin{equation}\label{eq:fold-change-variables}
    \begin{aligned}
        \int_\Omega \Phi(T_{-H}F_Hx)\dd \nu &= \int_{\Omega\cap H}\Phi(T_{-H}x)\dd\nu+\int_{\Omega\cap(\M_\kappa^n\setminus\overline{H})}\Phi(T_{-H}R_Hx)\dd\nu \\
                                            &= \bigg(\int_{\Omega\cap H}+\int_{R_H(\Omega)\cap H}\bigg)\Phi(T_{-H}x)\dd\nu \\
                                            &= \int_{\M_\kappa^n}\Phi(x)\alpha_H(x)\dd \nu
    \end{aligned}
\end{equation}
for every integrable scalar- or matrix-valued function $\Phi$. In particular, 
\begin{equation}\label{eq:volume-equality-with-folded-measure}
    \int_{\M_\kappa^n}\alpha_H\dd\nu = \vert \Omega \vert.
\end{equation}
Also by isometry, the same change of variables applies to the energy tensor of a Sobolev function composed with $T_{-H}$ and $F_{H}$. Indeed, for any $\Phi_i \in H_{\rm{loc}}^1(\M_\kappa^n)$, we have
\begin{equation*}
    \nabla(\Phi_i \circ T_{-H} \circ F_H)(x) = dF_H(x)^* dT_{-H}(F_Hx)^* \nabla\Phi_i(T_{-H}F_Hx).
\end{equation*}
Since $T_{-H}$ is an isometry and $F_H$ is an isometry on each side of $\partial H$, 
\begin{equation*}
    \big\langle \nabla(\Phi_i \circ T_{-H} \circ F_H)(x),\nabla(\Phi_j \circ T_{-H} \circ F_H)(x) \big\rangle = \big\langle \nabla\Phi_i(T_{-H}F_Hx),\nabla\Phi_j(T_{-H}F_Hx) \big\rangle.
\end{equation*}
Using \eqref{eq:fold-change-variables}, we obtain
\begin{equation}\label{eq:fold-change-variables-nabla^2}
    \begin{aligned}
        \int_\Omega \big\langle \nabla(\Phi_i \circ T_{-H} \circ F_H)(x),\nabla(\Phi_j \circ T_{-H} \circ F_H)(x) \big\rangle \dd\nu = \int_{\M_\kappa^n} \big\langle \nabla\Phi_i(x),\nabla\Phi_j(x) \big\rangle\alpha_H(x)\dd\nu.
    \end{aligned}
\end{equation}

Let $u_0(={\rm{const.}})$ and $u_1$ be orthonormal eigenfunctions corresponding to $\mu_0(\Omega)$ and $\mu_1(\Omega)$ respectively. 
The purpose of the folding construction is to choose $H$ and a center so that the component functions $P_i(x)$ of
\begin{equation}\label{eq:folded-test-function-packet-P}
    \mathcal{P}(x) = \mathcal{G}\big((T_{-H} \circ F_H)(x)\big)
\end{equation}
satisfy
\begin{equation}\label{eq:double-orthogonality}
    \int_\Omega P_i(x) u_0(x)\dd \nu = \int_\Omega P_i(x) u_1(x)\dd \nu = 0,
    \qquad i=1,\dots,n.
\end{equation}
Without any change of notation, $\mathcal{G}(r,\theta)=g(r)\theta$ is the field defined by \eqref{eq:radial-field}.

\subsection{Euclidean folding construction}
First, by translating $\Omega$, we may suppose the Weinberger center of $\Omega$ lies at the origin, so that
\begin{equation}\label{eq:Weinberger-center-condition}
    \int_\Omega \mathcal{G}(x)\dd\nu=0.
\end{equation}
The existence of such a center of mass above was proved by Weinberger \cite{Weinberger1956}. Also see \cite[Corollary~2]{Laugesen2021WeinbergerCenter}.

For $p\in\Sph^{n-1}$ and $t\ge0$, define
\begin{equation*}
    H:=H_{p,t}=\{x\in\R^n:x\cdot p<t\},\qquad R_H(x):=R_{p,t}(x)=x-2(x\cdot p-t)p,
\end{equation*}
be a geodesic half-space and the corresponding reflection, respectively. Let $F_H$ be the fold map defined in \eqref{eq:general-fold} associated with $H_{p,t}$. More precisely,
\begin{equation*}
    F_H(x):=F_{p,t}(x)=
    \begin{cases}
        \ x,&x\in H_{p,t},\\
        \ R_{p,t}(x),&x\notin H_{p,t}.
    \end{cases}
\end{equation*}
For any $(p,t)\in\mathbb{S}^{n-1}\times[0,\infty)$, define the translation by 
\[
    T_{-H}(x):=T_{-c_{p,t}}(x)=x-c_{p,t},
\]
where $c_{p,t}$ is the center of the folded measure such that
\begin{equation*}
    \int_\Omega \mathcal{G}\big(F_{p,t}x-c_{p,t}\big)\dd\nu = 0.
\end{equation*}
The existence and uniqueness of \(c_{p,t}\), together with its continuous dependence on \((p,t)\), follow from \cite[Corollary~3]{Laugesen2021WeinbergerCenter}. Furthermore, $c_{p,t}\in H_{p,t}$ \cite[Lemma~5.1]{FreitasLaugesen2022}.

Define a vector field
\[
    W(p,t):=\int_{\Omega}\mathcal G\bigl(F_{p,t}x-c_{p,t}\bigr)u_1(x)\dd\nu,\qquad p\in\Sph^{n-1}, t\ge0.
\]
Obviously, $W$ is continuous since $F_{p,t}$ and $c_{p,t}$ depend continuously on $(p,t)$. At \(t=0\), the reflection commutativity $\mathcal{G} \circ R_{p,0}=R_{p,0} \circ \mathcal{G}$ gives
\[
    c_{-p,0}=R_{p,0}(c_{p,0}),\qquad W(-p,0)=R_{p,0}W(p,0).
\]
For sufficiently large \(\tau>0\), one has \(\Omega\subset H_{p,\tau}\) for every \(p\in\mathbb S^{n-1}\) since $\overline{\Omega}$ is compact. 
Hence, $c_{p,\tau}=0$ by \eqref{eq:Weinberger-center-condition} and \(W(p,\tau)\equiv\int_\Omega\mathcal{G}u_1\dd\nu\) is independent of \(p\). 
If \(W\) had no zero in \(\mathbb S^{n-1}\times[0,\tau]\), then
\[
    \Phi(p,t):=\frac{W(p,t)}{|W(p,t)|}
\]
would define a homotopy of maps \(\mathbb S^{n-1}\to\mathbb S^{n-1}\). Petrides' symmetry principle \cite[Claim~3]{Petrides2014} gives
\[
    \deg\Phi(\,\cdot\,,0)\neq0,
\]
whereas \(\Phi(\,\cdot\,,\tau)\) is constant and therefore has degree zero. This contradicts the homotopy invariance of degree; see \cite[Chapter~III]{OutereloRuiz2009}. Thus, for some \((p,t)\in\mathbb S^{n-1}\times[0,\tau]\), we have
\[
    \int_{\Omega} \mathcal G\bigl(F_{p,t}x-c_{p,t}\bigr)u_1(x)\dd\nu=0.
\]

Hence, there exists a half-space $H_{p,t}$ for which \eqref{eq:double-orthogonality} holds. The corresponding fold density $\alpha_{H_{p,t}}$, defined by \eqref{eq:folded-measure}, satisfies \eqref{eq:fold-change-variables}, \eqref{eq:volume-equality-with-folded-measure}, and \eqref{eq:fold-change-variables-nabla^2} on $\mathbb{R}^n$. 
This degree-theoretic folding argument follows the construction in \cite[Section~5]{FreitasLaugesen2022}. 
An alternative construction can be found in \cite{BucurHenrot2019}.

\subsection{Hyperbolic folding construction}
Use the Poincar\'e ball model
\[
    \Hyp^n=\mathbb B^n,\qquad g_{\Hyp^n}=\dd r^2+\sinh(r)^2g_{\mathbb{S}^{n-1}}=\frac{4|\dd x|^2}{(1-|x|^2)^2},
\]
where $x=\exp_o(r\theta)$ and $r=d_{\mathbb{B}^n}(o,x)=\log\big(\frac{1+\vert x\vert}{1-\vert x\vert}\big), \theta\in\mathbb{S}^{n-1}$. For $x\in\mathbb{B}^n, y\in\overline{\mathbb{B}^n}$, define the M\"obius transformation by
\begin{equation*}
    \begin{aligned}
        T_x:\overline{\mathbb{B}^n}&\to\overline{\mathbb{B}^n}\\
         y&\mapsto T_x(y),
    \end{aligned}
\end{equation*}
where 
\begin{equation*}
T_x(y)=\frac{(1+2x\cdot y+|y|^2)x+(1-|x|^2)y}{1+2x\cdot y+|x|^2|y|^2}.	
\end{equation*}
Then $T_x$ is a hyperbolic isometry with $T_x(0)=x, T_0(y)=y$, and $T_x^{-1}=T_{-x}$. It maps $\mathbb{B}^n$ to itself and $\partial\mathbb{B}^n$ to itself. For the general theory of M\"obius transformations of the Poincar\'e ball, see \cite{Ahlfors1981}.

By replacing $\Omega$ by $T_x(\Omega)$ for some $x$, we may assume that \eqref{eq:Weinberger-center-condition} still holds in $\mathbb{B}^n$. The proof of existence of such a "hyperbolic Weinberger center-of-mass" can be found in \cite{Chavel1980}, \cite[Theorem~6.1]{BenguriaLinde2007} and \cite[Corollary~2]{Laugesen2021HerschSzegoCenter}. 

For $p\in\Sph^{n-1}$ and $t\in[0,1)$, put
\begin{equation*}
    H_{p,0}=\{y\in\mathbb B^n:y\cdot p<0\},\qquad H_{p,t}=T_{tp}(H_{p,0}),
\end{equation*}
and define the corresponding reflection by
\begin{equation*}
    R_{p,t}=T_{tp}\circ R_{p,0}\circ T_{-tp},\qquad R_{p,0}(y)=y-2(y\cdot p)p.
\end{equation*}
Therefore, $H:=H_{p,t}$ is a geodesic half-space in $\Hyp^n$, and $R_H:=R_{p,t}$ is the reflection across $\partial H_{p,t}$. Let $F_H:=F_{p,t}$ be the associated fold map, that is
\begin{equation*}
    F_{p,t}(x)=
    \begin{cases}
        \ x,&x\in H_{p,t},\\
        \ R_{p,t}(x),&x\notin H_{p,t}.
    \end{cases}
\end{equation*}

The following procedure is analogous to that in the Euclidean space above. Let $c_{p,t}\in\mathbb B^n$ be the center of the folded measure, and define the centering isometry by $T_{-H}:=T_{-c_{p,t}}$, such that
\begin{equation*}
    \int_{\Omega}\mathcal{G}\big(T_{-c_{p,t}}\circ F_{p,t}(x)\big)\dd\nu=0.
\end{equation*}
The existence and uniqueness of \(c_{p,t}\), together with its continuous dependence on \((p,t)\), follow from \cite[Corollary~3]{Laugesen2021HerschSzegoCenter}.

Define a vector field
\[
    W(p,t):=\int_{\Omega}\mathcal{G}\big(T_{-c_{p,t}}\circ F_{p,t}(x)\big)u_1(x)\dd\nu,\qquad p\in\Sph^{n-1}, t\in[0,1).
\]
Obviously, $W$ is continuous since $F_{p,t}$ and $c_{p,t}$ depend continuously on $p,t$. At \(t=0\), the reflection commutativity $\mathcal{G} \circ R_{p,0}=R_{p,0} \circ \mathcal{G}$ and identity
\begin{equation*}
    T_{R_{p,0}x} \circ R_{p,0} = R_{p,0} \circ T_{x},\qquad \forall x\in\mathbb{B}^n, p\in\mathbb{S}^{n-1},
\end{equation*}
give
\[
    c_{-p,0}=R_{p,0}(c_{p,0}),\qquad W(-p,0)=R_{p,0}W(p,0).
\]

For \(\tau\in(0,1)\) sufficiently close to \(1\), one has \(\Omega\subset H_{p,\tau}\) for every \(p\in\mathbb S^{n-1}\) since $\overline{\Omega}$ is compact. 
Hence, $c_{p,\tau}=0$ by \eqref{eq:Weinberger-center-condition} and \(W(p,\tau)\equiv\int_\Omega\mathcal{G}u_1\dd\nu\) is independent of \(p\). 
By an argument analogous to the preceding Euclidean proof, there is some \((p,t)\in\mathbb S^{n-1}\times[0,\tau]\), such that
\[
    \int_{\Omega}\mathcal{G}\big(T_{-c_{p,t}}\circ F_{p,t}(x)\big)u_1(x)\dd\nu=0.
\]

Hence, there exists a half-space $H_{p,t}$ such that \eqref{eq:fold-change-variables}, \eqref{eq:volume-equality-with-folded-measure}, \eqref{eq:fold-change-variables-nabla^2} and \eqref{eq:double-orthogonality} hold on $\mathbb{B}^n$. This argument follows from \cite[Section~7]{FreitasLaugesen2022}.

\subsection{Spherical folding construction}

Regard the unit sphere as
\[
    \Sph^n=\{x\in\R^{n+1}:|x|=1\}, \qquad g_{\Sph^n}=\dd r^2+\sin^2(r)g_{\mathbb S^{n-1}}.
\]
Fix \(o=e_{n+1}\). In geodesic polar coordinates centered at \(o\),
\[
    x=\cos(r)o+\sin(r)\theta, \qquad r=d_{\Sph^n}(o,x), \quad \theta\in\mathbb S^{n-1}\subset T_o\Sph^n.
\]

For \(a\in\Sph^n\), define
\[
    H:=H_a=\{x\in\Sph^n:x\cdot a<0\}, \qquad R_H(x):=R_a(x)=x-2(x\cdot a)a.
\]
Thus \(H_a\) is a geodesic hemisphere and \(R_a\) is the reflection across its totally geodesic boundary. The associated fold map $F_H:=F_a$ is defined by
\[
    F_a(x)=
    \begin{cases}
        \ x,      &x\in H_a,\\
        \ R_a(x), &x\notin H_a.
    \end{cases}
\]

For \(z\in\Sph^n\), write
\[
    y=\cos(t)z+\sin(t)\omega, \qquad t=d_{\Sph^n}(z,y),\quad \omega\in\mathbb S^{n-1}\subset T_z\Sph^n,
\]
and define the radial field centered at \(z\) analogously to \eqref{eq:radial-field} by
\[
    \mathcal G_z(y):=g(t)\omega.
\]
Here \(g\) satisfies \eqref{eq:radial-ode} and is extended to \([0,\pi]\) by \eqref{eq:g-extension}. Equivalently, setting
\[
    \pi_z(y):=y-(y\cdot z)z, \qquad \gamma(s):=\frac{g(\arccos s)}{\sqrt{1-s^2}},
\]
we have
\[
    \mathcal G_z(y)=\gamma(z\cdot y)\pi_z(y).
\]
Moreover, \(\gamma\) extends to a continuous even function on \([-1,1]\). Let
\[
    \Gamma(s):=\int_0^s\gamma(q)\dd q
\]
and, for \(\rho\in L^1(\Sph^n)\), define
\[
    E_\rho(a,z):=\int_{\Sph^n}\rho(x)\Gamma\bigl(z\cdot F_a(x)\bigr)\dd\nu.
\]
Its spherical gradient with respect to \(z\) is
\[
    \nabla_zE_\rho(a,z)=\int_{\Sph^n}\rho(x)\gamma(z\cdot F_ax)\pi_z(F_ax)\dd\nu=\int_{\Sph^n}\rho(x)\mathcal G_z\bigl(F_a(x)\bigr)\dd\nu.
\]

Apply Theorem~5 in \cite{BucurMartinetNahon2025} with $\rho=\1_\Omega, \sigma=u_1\1_\Omega$.
There exist \(a,z\in\Sph^n\) such that \(z\) is a common critical point
of \(E_\rho(a,\cdot)\) and \(E_\sigma(a,\cdot)\). Consequently,
\[
    \int_\Omega\mathcal G_z\bigl(F_a(x)\bigr)\dd\nu(x)=0, \qquad \int_\Omega\mathcal G_z\bigl(F_a(x)\bigr)u_1(x)\dd\nu(x)=0.
\]
This simultaneous choice of \(a\) and \(z\) replaces the uniqueness and continuous dependence of the center of the folded measure used in the Euclidean and hyperbolic cases. 
Its proof is based on a homotopy-invariant count of zeros modulo \(4\); see \cite[Sections~2--3]{BucurMartinetNahon2025}.

Choose a rotation \(T_{-z}\in SO(n+1)\) satisfying \(T_{-z}z=o\), and set
\[
    T_{-H}:=T_{-z}.
\]
The rotational equivariance of the radial field gives
\[
    \mathcal G\bigl(T_{-z}y\bigr) = T_{-z}\mathcal G_z(y).
\]
Therefore, we have
\begin{equation*}
    \int_\Omega\mathcal{G}\bigl(T_{-z} \circ F_a(x)\bigr)\dd\nu(x)=0, \qquad \int_\Omega\mathcal{G}\bigl(T_{-z} \circ F_a(x)\bigr)u_1(x)\dd\nu(x)=0.
\end{equation*}
These imply \eqref{eq:fold-change-variables}, \eqref{eq:volume-equality-with-folded-measure}, \eqref{eq:fold-change-variables-nabla^2} and \eqref{eq:double-orthogonality} hold on $\Sph^n$.

In particular, owing to the symmetry of $\mathbb{S}^{n}$, we modify the definition \eqref{eq:folded-measure} of $\alpha_H$ by
\begin{equation}\label{eq:temp-folded-measure-modification-S^n}
    \begin{aligned}
        \alpha_H :&= \1_{T_{-H}(\Omega\cap H)\cap\mathbb{S}^{n}_{+}} + \1_{T_{-H}(R_H(\Omega)\cap H)\cap\mathbb{S}^{n}_{+}} + \1_{-(T_{-H}(\Omega\cap H)\cap\mathbb{S}^{n}_{-})} + \1_{-(T_{-H}(R_H(\Omega)\cap H)\cap\mathbb{S}^{n}_{-})} \\
                  &= \1_{T_{-H}(\Omega\cap H)\cap\mathbb{S}^{n}_{+}} + \1_{T_{-H}(R_H(\Omega)\cap H)\cap\mathbb{S}^{n}_{+}} + \1_{-T_{-H}(\Omega\cap H)\cap\mathbb{S}^{n}_{+}} + \1_{-T_{-H}(R_H(\Omega)\cap H)\cap\mathbb{S}^{n}_{+}}. \\
    \end{aligned}
\end{equation}
where
\begin{equation*}
    \mathbb{S}^{n}_{+}=\{x\in\mathbb{S}^n: x\cdot e_{n+1}>0\} \qquad \text{and} \qquad \mathbb{S}^{n}_{-}=\{x\in\mathbb{S}^n: x\cdot e_{n+1}<0\}
\end{equation*}
are the upper and lower hemispheres, respectively. Since the equator $\mathbb{S}^{n}\setminus(\mathbb{S}^{n}_{+}\cup\mathbb{S}^{n}_{-})$ has measure zero, we may neglect it. 

Intuitively, since $T_{-H}(\Omega\cap H)$ and $T_{-H}(R_H(\Omega)\cap H)$ lie within the hemisphere $T_{-H}H$, we retain the measure on $T_{-H}H\cap\mathbb{S}^{n}_{+}$ unchanged and symmetrically push forward the measure on $T_{-H}H\cap\mathbb{S}^{n}_{-}$ to the antipodal points. As a result, the modified fold density $\alpha_H$ is supported on the closure of the upper hemisphere $\mathbb{S}^{n}_{+}$,
\begin{equation*}
    \operatorname{supp}(\alpha_H)\subset\overline{\mathbb{S}^n_+}.
\end{equation*}

Since $T_{-H}(\Omega\cap H)\cap\mathbb{S}^{n}_{+} \subset T_{-H}H$ while $-T_{-H}(\Omega\cap H)\cap\mathbb{S}^{n}_{+} \subset -T_{-H}H$, they are disjoint. Similarly, $T_{-H}(R_H(\Omega)\cap H)\cap\mathbb{S}^{n}_{+}$ and $-T_{-H}(R_H(\Omega)\cap H)\cap\mathbb{S}^{n}_{+}$ are also disjoint. Therefore, \eqref{eq:temp-folded-measure-modification-S^n} can be reduced to 
\begin{equation}\label{eq:folded-measure-modification-S^n}
    \alpha_H = \1_{\Omega_1\cap\mathbb{S}^{n}_{+}} + \1_{\Omega_2\cap\mathbb{S}^{n}_{+}},
\end{equation}
where
\begin{equation*}
    \begin{aligned}
        \Omega_1 &:= \Big(T_{-H}(\Omega\cap H)\Big)      \sqcup \Big(-T_{-H}(\Omega\cap H)\Big), \\
        \Omega_2 &:= \Big(T_{-H}(R_H(\Omega)\cap H)\Big) \sqcup \Big(-T_{-H}(R_H(\Omega)\cap H)\Big).
    \end{aligned} 
\end{equation*}

Therefore, if the function $\Phi$ is symmetric on the sphere, that is $\Phi(x)=\Phi(-x)$, \eqref{eq:fold-change-variables} can be reduced to 
\begin{equation}\label{eq:fold-change-variables-modification-S^n}
    \int_\Omega \Phi(T_{-H}F_Hx)\dd\nu = \int_{\mathbb{S}^{n}_{+}}\Phi(x)\alpha_{H}(x)\dd\nu.
\end{equation}
Analogously, if $\Phi_i$ and $\Phi_j$ satisfy $\langle (\nabla\Phi_i)(x),(\nabla\Phi_j)(x) \rangle=\langle (\nabla\Phi_i)(-x),(\nabla\Phi_j)(-x) \rangle$, \eqref{eq:fold-change-variables-nabla^2} can be reduced to 
\begin{equation}\label{eq:fold-change-variables-nabla^2-modification-S^n}
    \begin{aligned}
        \int_\Omega \big\langle \nabla(\Phi_i \circ T_{-H} \circ F_H)(x),\nabla(\Phi_j \circ T_{-H} \circ F_H)(x) \big\rangle \dd\nu = \int_{\mathbb{S}^{n}_{+}} \big\langle \nabla\Phi_i(x),\nabla\Phi_j(x) \big\rangle\alpha_H(x)\dd\nu.
    \end{aligned}
\end{equation}
Moreover, 
\begin{equation}\label{eq:volume-equality-with-folded-measure-modification-S^n}
    \int_{\mathbb{S}^{n}_{+}} \alpha_H \dd\nu = \vert \Omega \vert.
\end{equation}
Here $\alpha_H$ in \eqref{eq:fold-change-variables-modification-S^n}, \eqref{eq:fold-change-variables-nabla^2-modification-S^n} and \eqref{eq:volume-equality-with-folded-measure-modification-S^n} is the modified fold density defined by \eqref{eq:folded-measure-modification-S^n}.

\section{Proof of the inequality in \cref{thm:main}}\label{sec:Proof of main theorem}
With the folded trial fields constructed in the preceding section, we now prove the sharp inequality \eqref{eq:main}. The argument combines the shifted Ritz reduction with a global matrix-valued mass transplantation comparing the fold density \(\alpha_H\) with the saturated density \(2\cdot\1_{B_R^\kappa}\), followed by the trace-free matrix inequality from \cref{sec:Variational and matrix tools}.

\subsection{Fold density and the trace-free defect matrix}
Let $\alpha_H$ be the fold density defined by \eqref{eq:folded-measure} when $\kappa\le0$, and be the modified fold density defined by \eqref{eq:folded-measure-modification-S^n} when $\kappa=1$. Hence $\alpha_H$ is compactly supported, and $0\le\alpha_H\le2$. 

Define the matrix
\begin{equation}\label{eq:Z-definition}
    \mathbf{Z}:=\int_{B^\kappa_{L_\kappa}}\Big(\alpha_H(x)-2\cdot\1_{B^\kappa_R}(x)\Big)\theta\theta^\transpose\dd\nu.
\end{equation}
Using \eqref{eq:volume-equality-with-folded-measure} when $\kappa\leq0$ and \eqref{eq:volume-equality-with-folded-measure-modification-S^n} when $\kappa=1$, together with $\vert\Omega\vert=2\vert B_R^\kappa\vert$, we obtain
\begin{equation}\label{eq:global-mass-balance}
    \int_{B^\kappa_{L_\kappa}}\Big(\alpha_H(x)-2\cdot\1_{B^\kappa_R}(x)\Big)\dd\nu=0.
\end{equation}
Since $\tr(\theta\theta^\transpose)=1$, 
\begin{equation}\label{eq:Z-trace-free}
    \tr\mathbf Z=0.
\end{equation}

\subsection{Folded trial space and the shifted Ritz bound}
Let \(P_i\) denote the \(i\)-th component of the field \(\mathcal P\) which is defined by \eqref{eq:folded-test-function-packet-P}. Since $\mathcal{G}, T_{-H}$ and $F_H$ are locally Lipschitz, $P_1,\cdots,P_n \in H^1(\Omega)$. Moreover, the orthogonality condition \eqref{eq:double-orthogonality} holds on all three space forms, in view of the discussion in \cref{sec:Folding constructions on space forms}.

Define matrices $\mathbf{M}$ and $\mathbf{K}$ by
\begin{equation*}
    \mathbf{M}_{ij}:=\int_\Omega P_iP_j\dd\nu, \qquad \mathbf{K}_{ij}:=\int_\Omega \langle\nabla P_i,\nabla P_j\rangle\dd\nu.
\end{equation*}
If $\kappa=1$, let $\Phi=G_iG_j$. Then
\begin{equation*}
    \Phi(-x) = g(\pi-r)(-\theta_i) \cdot g(\pi-r)(-\theta_j) = g(r)\theta_i \cdot g(r)\theta_j = \Phi(x).
\end{equation*}
Similarly, let $\Phi_i=G_i, \Phi_j=G_j$, 
\begin{equation*}
    \langle (\nabla\Phi_i)(-x),(\nabla\Phi_j)(-x) \rangle = \langle (\nabla\Phi_i)(x),(\nabla\Phi_j)(x) \rangle.
\end{equation*}
Therefore, by using \eqref{eq:GG-and-nablaGnablaG}, \eqref{eq:fold-change-variables}, \eqref{eq:fold-change-variables-nabla^2}, \eqref{eq:fold-change-variables-modification-S^n} and \eqref{eq:fold-change-variables-nabla^2-modification-S^n}, 
\begin{equation}\label{eq:expression-for-M-and-K}
    \begin{aligned}
        \mathbf{M} &= \int_{B^\kappa_{L_\kappa}}       g(r)^2 \theta\theta^\transpose                                                                             \cdot \alpha_H(x) \dd\nu, \\
        \mathbf{K} &= \int_{B^\kappa_{L_\kappa}} \bigg[g'(r)^2\theta\theta^\transpose + \frac{g(r)^2}{\sn_\kappa(r)^2}\Big(\Id-\theta\theta^\transpose\Big)\bigg] \cdot \alpha_H(x) \dd\nu.
    \end{aligned}
\end{equation}

The next proposition converts the folded trial space into the desired shifted trace estimate.
\begin{proposition}\label{prop:Sigma1/mu-geq-trK^-1M}
    The matrices $\mathbf M$ and $\mathbf K$ defined above satisfy
    \begin{equation*}
        \sum_{j=2}^{n+1}\frac1{\mu_j(\Omega)}\geq\tr(\mathbf K^{-1}\mathbf M).
    \end{equation*}
\end{proposition}
\begin{proof}
    By \cref{lem:shifted-ritz}, it suffices to prove that \(P_1,\ldots,P_n\) are linearly independent in \(L^2(\Omega)\), and $\mathbf K$ is positive definite.

    To prove that \(P_1,\ldots,P_n\) are linearly independent, we show that $\mathbf M$ is positive definite since $\mathbf{M}$ is the Gram matrix.
    Fix \(0\neq\xi\in\R^n\) and set $P_\xi:=\sum_{i=1}^n\xi_iP_i$. By \eqref{eq:expression-for-M-and-K},
    \begin{equation}\label{eq:M-positive-definite}
        \xi^\transpose\mathbf M\xi = \int_{B^\kappa_{L_\kappa}} g(r)^2(\xi\cdot\theta)^2 \alpha_H(x)\dd\nu.
    \end{equation}
    The extension \eqref{eq:g-extension} satisfies $g(r)>0$ on $0<r<L_\kappa$. Moreover, since \(\xi\neq0\) and \(n\geq2\),
    \[
        \sigma\bigl(\{\theta\in\Sph^{n-1}:\xi\cdot\theta=0\}\bigr)=0.
    \]
    Thus \(g(r)^2(\xi\cdot\theta)^2>0\) for almost every \((r,\theta)\), or equivalently, for $\nu$-a.e. $x$. Since
    \[
        \int_{B^\kappa_{L_\kappa}}\alpha_H(x)\dd\nu=|\Omega|>0,
    \]
    equation \eqref{eq:M-positive-definite} gives
    \[
        \xi^\transpose\mathbf M\xi>0.
    \]
    Hence \(\mathbf M\) is positive definite, which implies that \(P_1,\ldots,P_n\) are linearly independent in \(L^2(\Omega)\).

    Similarly, \eqref{eq:expression-for-M-and-K} yields
    \begin{equation*}
        \begin{aligned}
            \xi^\transpose\mathbf K\xi = \int_{B^\kappa_{L_\kappa}} \bigg[g'(r)^2(\xi\cdot\theta)^2+\frac{g(r)^2}{\sn_\kappa(r)^2}\Big(|\xi|^2-(\xi\cdot\theta)^2\Big)\bigg]\alpha_H(x)\dd\nu.
        \end{aligned}
    \end{equation*}
    For \(0<r<L_\kappa\), $g^2/\sn_\kappa^2>0$. Furthermore,
    \[
        |\xi|^2-(\xi\cdot\theta)^2=0 \quad \Longleftrightarrow \quad \theta=\pm\frac{\xi}{|\xi|}.
    \]
    The set \(\{\pm\xi/|\xi|\}\subset\Sph^{n-1}\) has zero \((n-1)\)-dimensional measure. Consequently, the integrand above is strictly positive for almost every \((r,\theta)\). Since the fold density has total mass \(|\Omega|>0\), it follows that
    \[
        \xi^\transpose\mathbf K\xi>0.
    \]
    Therefore \(\mathbf K\) is positive definite.
\end{proof}

\subsection{Global transplantation for the mass matrix}
To apply \cref{lem:trace-free-matrix-Jensen-inequality}, we next construct the matrix comparison. Set
\begin{equation*}
    a=\frac{2A_R}{n}, \qquad c=g(R)^2, \qquad d=\frac{g(R)^2}{\sn_\kappa(R)^2}.
\end{equation*}
By symmetry, 
\begin{equation}\label{eq:integral-of-thetatheta^T}
    \int_{\Sph^{n-1}}\theta\theta^\transpose\dd\sigma(\theta) = \frac{\vert\Sph^{n-1}\vert}{n}\Id.
\end{equation}

\begin{proposition}\label{prop:M-geq-aI+cZ}
    The matrices $\mathbf M$ and $\mathbf Z$ defined above satisfy
    \begin{equation*}
        \mathbf M\succeq a\Id+c\mathbf Z.
    \end{equation*}
    Equality holds if and only if $\alpha_H=2\cdot\1_{B^\kappa_R}$ for $\nu$-a.e. $x\in B^\kappa_{L_\kappa}$.
\end{proposition}
\begin{proof}
    Using \eqref{eq:integral-of-thetatheta^T}, 
    \begin{equation*}
        \begin{aligned}
            a\Id = \int_{B^\kappa_{L_\kappa}} 2\cdot\1_{B^\kappa_R}(x) g(r)^2 \theta\theta^\transpose \dd\nu.
        \end{aligned}
    \end{equation*}
    Combining \eqref{eq:Z-definition} and \eqref{eq:expression-for-M-and-K} yields
    \begin{equation}\label{eq:temp-M-geq-aI+cZ}
        \begin{aligned}
            \mathbf{M}-a\Id-c\mathbf{Z} = \int_{B^\kappa_{L_\kappa}} \Big(g(r)^2-g(R)^2\Big)\Big(\alpha_H(x)-2\cdot\1_{B^\kappa_R}(x)\Big) \theta\theta^\transpose \dd\nu.
        \end{aligned}
    \end{equation}
    By \cref{lem:radial-monotonicity}, $g^2$ is non-decreasing on $[0,L_\kappa)$. Consequently, $g(r)^2 < g(R)^2$ and $\alpha_H(x) \le 2 = 2\cdot \mathbf{1}_{B^\kappa_R}$ on $B^\kappa_R$. On $B^\kappa_{L_\kappa}\setminus B^\kappa_R$, we instead have $g(r)^2 = g(R)^2$ and $\alpha_H(x) \ge 0 = 2\cdot \mathbf{1}_{B^\kappa_R}$.
    Therefore,
    \begin{equation*}
        \Big(g(r)^2-g(R)^2\Big)\Big(\alpha_H(x)-2\cdot\1_{B^\kappa_R}(x)\Big) \ge 0, \qquad \text{for $\nu$-a.e. } x \in B^\kappa_{L_\kappa}.
    \end{equation*}
    Since $\theta\theta^\transpose$ is positive semidefinite, we obtain
    \begin{equation*}
        \mathbf{M}-a\Id-c\mathbf{Z} \succeq \mathbf{0}.
    \end{equation*}

    We now turn to the equality case. If $\alpha_H=2\cdot\1_{B^\kappa_R}$ for $\nu$-a.e. $x\in B^\kappa_{L_\kappa}$, equality follows immediately from \eqref{eq:temp-M-geq-aI+cZ}. 
    Conversely, if $\mathbf{M}=a\Id+c\mathbf{Z}$, taking the trace of \eqref{eq:temp-M-geq-aI+cZ} yields
    \begin{equation*}
        \int_{B^\kappa_{L_\kappa}} \Big(g(r)^2-g(R)^2\Big)\Big(\alpha_H(x)-2\cdot\1_{B^\kappa_R}(x)\Big) \dd\nu = 0.
    \end{equation*}
    This implies that
    \begin{equation*}
        \alpha_H(x) = 2, \qquad \text{for $\nu$-a.e. } x\in B^\kappa_{R},
    \end{equation*}
    since $g(r)^2<g(R)^2$ on $B^\kappa_R$. Together with \eqref{eq:global-mass-balance}, it follows that
    \begin{equation*}
        \alpha_H(x) = 2\cdot\1_{B^\kappa_R}(x), \qquad \text{for $\nu$-a.e. } x\in B^\kappa_{L_\kappa}.
    \end{equation*}
\end{proof}

\subsection{Global transplantation for the energy matrix}
We next estimate the energy matrix from above by an analogous global comparison.
\begin{proposition}\label{prop:K-leq-lambdaaI-dZ}
    The matrices $\mathbf K$ and $\mathbf Z$ defined above satisfy
    \begin{equation*}
        \mathbf K\preceq\lambda a\Id-d\mathbf Z.
    \end{equation*}
    Equality holds if and only if $\alpha_H=2\cdot\1_{B^\kappa_R}$ for $\nu$-a.e. $x\in B^\kappa_{L_\kappa}$.
\end{proposition}
\begin{proof}
    Using \eqref{eq:energy-identity} and \eqref{eq:integral-of-thetatheta^T}, 
    \begin{equation*}
        \begin{aligned}
            \lambda a\Id &= \frac{2}{n}\Big(Q_R+(n-1)H_R\Big)\Id \\
                         &= \int_{B^\kappa_{L_\kappa}} 2\cdot\1_{B^\kappa_R}(x)\bigg[g'(r)^2\theta\theta^\transpose + \frac{g(r)^2}{\sn_\kappa(r)^2}\Big(\Id-\theta\theta^\transpose\Big)\bigg]\dd\nu.
        \end{aligned}
    \end{equation*}
    On the other hand, using \eqref{eq:global-mass-balance}, we may rewrite \eqref{eq:Z-definition} as
    \begin{equation*}
        -d\mathbf{Z}=\int_{B^\kappa_{L_\kappa}}\frac{g(R)^2}{\sn_\kappa(R)^2}\Big(\alpha_H(x)-2\cdot\1_{B^\kappa_R}(x)\Big)\Big(\Id-\theta\theta^\transpose\Big)\dd\nu.
    \end{equation*}
    Therefore, combining these identities with \eqref{eq:expression-for-M-and-K} gives
    \begin{equation}\label{eq:temp-K-leq-lambdaaI-dZ}
        \begin{aligned}
            \lambda a\Id-d\mathbf{Z}-\mathbf{K} &= \int_{B^\kappa_{L_\kappa}} \Bigg[ \Big(2\cdot\1_{B^\kappa_R}(x)-\alpha_H(x)\Big)g'(r)^2\theta\theta^\transpose \\
                                                &\qquad\qquad\qquad + \Big(2\cdot\1_{B^\kappa_R}(x)-\alpha_H(x)\Big)\bigg(\frac{g(r)^2}{\sn_\kappa(r)^2}-\frac{g(R)^2}{\sn_\kappa(R)^2}\bigg)\Big(\Id-\theta\theta^\transpose\Big) \Bigg]\dd\nu.
        \end{aligned}
    \end{equation}
    By \cref{lem:radial-monotonicity}, $g^2/\sn_\kappa^2$ is strictly decreasing on $(0,L_\kappa)$. Consequently, 
    \begin{equation*}
        \frac{g(r)^2}{\sn_\kappa(r)^2}>\frac{g(R)^2}{\sn_\kappa(R)^2}, \qquad 2\cdot\1_{B^\kappa_R}(x)=2\ge\alpha_H(x), \qquad g'(r)^2>0, \qquad\quad \text{ on } B^\kappa_R,
    \end{equation*}
    and
    \begin{equation*}
        \frac{g(r)^2}{\sn_\kappa(r)^2}<\frac{g(R)^2}{\sn_\kappa(R)^2}, \qquad 2\cdot\1_{B^\kappa_R}(x)=0\le\alpha_H(x), \qquad g'(r)^2=0, \qquad\quad \text{ on } B^\kappa_{L_\kappa}\setminus \overline{B^\kappa_R}.
    \end{equation*}
    Therefore, for $\nu$-a.e. $x \in B^\kappa_{L_\kappa}$, 
    \begin{equation*}
        \Big(2\cdot\1_{B^\kappa_R}(x)-\alpha_H(x)\Big)g'(r)^2 \ge 0, \qquad \Big(2\cdot\1_{B^\kappa_R}(x)-\alpha_H(x)\Big)\bigg(\frac{g(r)^2}{\sn_\kappa(r)^2}-\frac{g(R)^2}{\sn_\kappa(R)^2}\bigg) \ge 0.
    \end{equation*}
    Since $\theta\theta^\transpose$ and $\Id-\theta\theta^\transpose$ are positive semidefinite, we obtain
    \begin{equation*}
        \lambda a\Id-d\mathbf{Z}-\mathbf{K} \succeq \mathbf{0}.
    \end{equation*}

    We now turn to the equality case. If $\alpha_H=2\cdot\1_{B^\kappa_R}$ for $\nu$-a.e. $x\in B^\kappa_{L_\kappa}$, equality follows immediately from \eqref{eq:temp-K-leq-lambdaaI-dZ}. 
    Conversely, if $\mathbf{K}=\lambda a\Id-d\mathbf{Z}$, taking the trace of \eqref{eq:temp-K-leq-lambdaaI-dZ} yields
    \begin{equation*}
        \int_{B^\kappa_{L_\kappa}} \Bigg[ \Big(2\cdot\1_{B^\kappa_R}(x)-\alpha_H(x)\Big)g'(r)^2 + (n-1)\Big(2\cdot\1_{B^\kappa_R}(x)-\alpha_H(x)\Big)\bigg(\frac{g(r)^2}{\sn_\kappa(r)^2}-\frac{g(R)^2}{\sn_\kappa(R)^2}\bigg) \Bigg]\dd\nu = 0.
    \end{equation*}
    This implies that
    \begin{equation*}
        \alpha_H(x) = 2\cdot\1_{B^\kappa_R}(x), \qquad \text{for $\nu$-a.e. } x\in B^\kappa_{L_\kappa},
    \end{equation*}
    since $g^2/\sn_\kappa^2$ is strictly decreasing.
\end{proof}

We are now in a position to combine the Ritz reduction with the matrix comparison and complete the proof of the inequality.
\subsection{Completion of the inequality proof}
Combining \eqref{eq:Z-trace-free} with \cref{prop:M-geq-aI+cZ} and \cref{prop:K-leq-lambdaaI-dZ}, \cref{lem:trace-free-matrix-Jensen-inequality} yields
\begin{equation*}
    \tr(\mathbf K^{-1}\mathbf M)\geq\frac n\lambda,
\end{equation*}
where $\lambda=\mu_1(B^\kappa_R)$. Applying \cref{prop:Sigma1/mu-geq-trK^-1M}, the desired inequality 
\begin{equation*}
    \sum_{j=2}^{n+1}\frac1{\mu_j(\Omega)} \geq \frac{n}{\mu_1(B^\kappa_R)} = \frac{n}{\mu_2(B^\kappa_R \sqcup B^\kappa_R)}
\end{equation*}
follows. This proves the inequality \eqref{eq:main} in \cref{thm:main}. It remains to analyze the equality case, which is carried out in the next section.

\section{Equality and geometric rigidity}\label{sec:Equality and geometric rigidity}
We first verify that the proposed extremizers attain equality and then prove the converse rigidity statement.

If $\Omega=B_1\sqcup B_2$, where $B_1,B_2$ are disjoint equal geodesic balls, then $\mu_0(\Omega)=\mu_1(\Omega)=0$ and
\[
    \mu_2(\Omega)=\cdots=\mu_{n+1}(\Omega)=\mu_1(B_1),
\]
so equality holds. Indeed, since \(\mu_1(B)\) has multiplicity \(n\), the eigenvalue \(\mu_2(B\sqcup B)\) has multiplicity \(2n\).

Conversely, suppose equality holds in \cref{thm:main}. Since the proof above gives
\begin{equation*}
    \sum_{j=2}^{n+1}\frac{1}{\mu_j(\Omega)} \ge \tr(\mathbf K^{-1}\mathbf M) \ge \frac n\lambda,
\end{equation*}
both inequalities are equalities. By \cref{lem:trace-free-matrix-Jensen-inequality}, we obtain
\begin{equation*}
    \mathbf{Z}=\mathbf{0}, \qquad \mathbf{M}=a\Id, \qquad \mathbf{K}=\lambda a\Id.
\end{equation*}
Therefore, the equalities in \cref{prop:M-geq-aI+cZ} and \cref{prop:K-leq-lambdaaI-dZ} hold, which imply
\begin{equation}\label{eq:alpha_H=2 1_BR}
    \begin{aligned}
        \alpha_{H} = 2\cdot\1_{B^\kappa_R} \ \ \qquad \text{for $\nu$-a.e. } x\in\M_\kappa^n.
    \end{aligned}
\end{equation}

\subsection{The Euclidean and hyperbolic cases}
Recall from \eqref{eq:folded-measure} that the fold density $\alpha_H$ is 
\begin{equation*}
    \alpha_H := \1_{T_{-H}(\Omega\cap H)} + \1_{T_{-H}(R_H(\Omega)\cap H)}.
\end{equation*}
By \eqref{eq:alpha_H=2 1_BR}, 
\begin{equation*}
    T_{-H}(\Omega\cap H) = T_{-H}(R_H(\Omega)\cap H) = B^\kappa_R \quad \text{up to a $\nu$-null set}, \qquad \kappa\le0.
\end{equation*}
This implies both $\Omega \cap H$ and $\Omega\cap(\M_\kappa^n\setminus\overline{H})$ agree almost everywhere with geodesic balls of radius $R$. 
Since $H$ and $\M_\kappa^n\setminus\overline{H}$ on opposite sides of $\partial H$ are disjoint, 
\begin{equation}\label{eq:Omega=B1 cup B2 a.e.}
    \Omega = B_1 \sqcup B_2 \quad \text{up to a $\nu$-null set},
\end{equation}
where $B_1$ and $B_2$ are two disjoint geodesic balls of radius $R$.

Since \(\Omega\) is open, \eqref{eq:Omega=B1 cup B2 a.e.} implies $\Omega\subset B_1\sqcup B_2$.
Indeed, otherwise, $\exists x\in\Omega\setminus(B_1\sqcup B_2)$ would have a neighborhood $B_\varepsilon^\kappa(x)$ whose intersection with \(\Omega\setminus(B_1\sqcup B_2)\) has positive volume, contradicting \eqref{eq:Omega=B1 cup B2 a.e.}.

Meanwhile, since \(\Omega\) is Lipschitz, \eqref{eq:Omega=B1 cup B2 a.e.} implies $B_1\sqcup B_2\subset\Omega$. 
Suppose that $\exists x\in(B_1\sqcup B_2)\setminus\Omega$. Since $\nu\bigl((B_1\sqcup B_2)\setminus\Omega\bigr)=0$, we must have \(x\in\partial\Omega\). As \(x\) is an interior point of \(B_1\cup B_2\), the exterior density property of the Lipschitz domain \(\Omega\) gives, for all sufficiently small \(\varepsilon>0\),
\[
    0<\nu\bigl(B_\varepsilon^\kappa(x)\setminus\Omega\bigr)\leq\nu\bigl((B_1\sqcup B_2)\setminus\Omega\bigr)=0,
\]
a contradiction.

Therefore,
\[
    \Omega=B_1 \sqcup B_2.
\]

\subsection{The spherical case}
Recall that the modified fold density $\alpha_H$ defined by \eqref{eq:folded-measure-modification-S^n} is
\begin{equation*}
    \alpha_H = \1_{\Omega_1\cap\mathbb{S}^{n}_{+}} + \1_{\Omega_2\cap\mathbb{S}^{n}_{+}},
\end{equation*}
where
\begin{equation*}
    \begin{aligned}
        \Omega_1 &:= \Big(T_{-H}(\Omega\cap H)\Big)      \sqcup \Big(-T_{-H}(\Omega\cap H)\Big), \\
        \Omega_2 &:= \Big(T_{-H}(R_H(\Omega)\cap H)\Big) \sqcup \Big(-T_{-H}(R_H(\Omega)\cap H)\Big).
    \end{aligned}
\end{equation*}
By \eqref{eq:alpha_H=2 1_BR},
\begin{equation}\label{eq:temp1-equality-on-S^n}
    \Omega_1\cap\mathbb{S}^n_+ = \Omega_2\cap\mathbb{S}^n_+ = B^1_R \quad\text{up to a $\nu$-null set}.
\end{equation}
Since $\Omega_1$ and $\Omega_2$ are antipodally symmetric, and since $B_R^1\subset\mathbb{S}^n_+$, it follows that
\begin{equation*}
    \Omega_1 = \Omega_2 = B_R^1\sqcup(-B_R^1) \quad \text{up to a $\nu$-null set}.
\end{equation*}

We first show that the centered folding interface $T_{-H}(\partial H)$ cannot meet the interior of either comparison ball. Since $\mathbf K=\lambda\mathbf M$, every nonzero function in $\Span\{P_1,\ldots,P_n\}$ has Rayleigh quotient $\lambda$. Since $\mu_{j+1}(\Omega)\leq\lambda$ for $j=1,\ldots,n$ and equality holds in the reciprocal sum, we have
\[
    \mu_2(\Omega)=\cdots=\mu_{n+1}(\Omega)=\lambda.
\]
The spectral decomposition on $\{u_0,u_1\}^{\perp}$ then shows that $\Span\{P_1,\ldots,P_n\}$ is contained in the eigenspace corresponding to $\lambda$. In particular, each $P_i$ is smooth in the interior of $\Omega$.

Suppose that $T_{-H}(\partial H)$ meets the interior of $B_R^1$ or $-B_R^1$. By \eqref{eq:temp1-equality-on-S^n} and antipodal symmetry, after undoing $T_{-H}$ and $R_H$, the domain $\Omega$ occupies both sides of the corresponding portion of $\partial H$ up to a $\nu$-null set. Since $\Omega$ is open and Lipschitz, the exterior density property excludes this portion from being a null crack. Hence, it is an interior hypersurface of $\Omega$.

Let $x$ be a point on this hypersurface, let $\eta$ be a unit normal to $\partial H$ at $x$, and put $y=T_{-H}x$. Since $F_H$ is the identity on one side of $\partial H$ and equals $R_H$ on the other side, while $dR_H(\eta)=-\eta$, we have
\begin{equation*}
    \left.\partial_\eta\mathcal{P}\right|_{+}=D\mathcal{G}(y)dT_{-H}(\eta), \qquad \left.\partial_\eta\mathcal{P}\right|_{-}=-D\mathcal{G}(y)dT_{-H}(\eta).
\end{equation*}
These normal derivatives are nonzero because $D\mathcal G$ is nonsingular in the interior of $B_R^1\sqcup(-B_R^1)$. Indeed, on $B_R^1$ its radial and tangential factors are $g'(r)>0$ and $g(r)/\sin r>0$, respectively, while the corresponding nonsingularity on $-B_R^1$ follows from antipodal symmetry. Therefore, $\mathcal{P}$ fails to be $C^1$ across an interior hypersurface of $\Omega$, contradicting the interior smoothness of its component functions $P_i$.

Consequently, the connected ball $B_R^1$ lies entirely on one side of the centered folding interface. After interchanging the two antipodal branches if necessary, we may suppose that
\begin{equation*}
    -T_{-H}(\Omega\cap H)\cap\mathbb{S}^n_+ = \varnothing \quad\text{up to a $\nu$-null set}.
\end{equation*}
It then follows from \eqref{eq:temp1-equality-on-S^n} that
\begin{equation}\label{eq:temp3-equality-on-S^n}
    T_{-H}(\Omega\cap H) = B_R^1 \quad\text{up to a $\nu$-null set}.
\end{equation}

Moreover, both $T_{-H}(\Omega\cap H)$ and $T_{-H}(R_H(\Omega)\cap H)$ lie in the same hemisphere $T_{-H}H$. Hence the preceding choice of antipodal branch and \eqref{eq:temp1-equality-on-S^n} also give
\begin{equation}\label{eq:temp4-equality-on-S^n}
    T_{-H}(R_H(\Omega)\cap H) = B_R^1 \quad\text{up to a $\nu$-null set}.
\end{equation}

By \eqref{eq:temp3-equality-on-S^n} and \eqref{eq:temp4-equality-on-S^n}, the two parts of $\Omega$ lying on the opposite sides of $\partial H$ agree almost everywhere with two disjoint geodesic balls of radius $R$. As in the Euclidean and hyperbolic cases, the openness of $\Omega$ and the Lipschitz exterior density property upgrade this almost-everywhere identity to an equality of domains. Therefore,
\begin{equation*}
    \Omega=B_1\sqcup B_2,
\end{equation*}
where $B_1$ and $B_2$ are two disjoint equal geodesic balls.

This completes the proof of \cref{thm:main}.

\section*{Acknowledgments}
The authors wish to express their sincere gratitude to Professor Richard S. Laugesen for his careful reading of the manuscript and for his valuable suggestion to pursue a global approach to the matrix estimates.
The authors also acknowledge the use of AI for assistance with English-language editing and polishing. 
The authors take full responsibility for the mathematical content and the final manuscript.


\bibliographystyle{amsalpha}
\bibliography{CY}

\vspace{1cm}
%

\end{document}